\documentclass[11pt]{amsart}
\usepackage[margin=1.35in]{geometry}
\usepackage{amscd,amsmath,amsxtra,amsthm,amssymb,stmaryrd,xr,mathrsfs,mathtools,enumerate,commath, comment}
\usepackage{stmaryrd}
\usepackage{xcolor}
\usepackage{commath}
\usepackage{comment}
\usepackage{tikz-cd}
\usepackage{longtable} 
\usepackage{pdflscape} 
\usepackage{booktabs}
\usepackage{hyperref}
\definecolor{vegasgold}{rgb}{0.77, 0.7, 0.35}
\definecolor{darkgoldenrod}{rgb}{0.72, 0.53, 0.04}
\definecolor{gold(metallic)}{rgb}{0.83, 0.69, 0.22}
\hypersetup{
 colorlinks=true,
 linkcolor=darkgoldenrod,
 filecolor=brown,      
 urlcolor=gold(metallic),
 citecolor=darkgoldenrod,
 pdftitle={Remarks on Hilbert's tenth problem and the Iwasawa theory of elliptic curves},
 }
\usepackage[all,cmtip]{xy}

\newcommand{\cF}{\mathcal{F}}

\newcommand{\cO}{\mathcal{O}}

\newcommand{\Z}{\mathbb{Z}}

\DeclareFontFamily{U}{wncy}{}
\DeclareFontShape{U}{wncy}{m}{n}{<->wncyr10}{}
\DeclareSymbolFont{mcy}{U}{wncy}{m}{n}
\DeclareMathSymbol{\Sh}{\mathord}{mcy}{"58}
\newcommand{\Q}{\mathbb{Q}}
\newcommand{\op}[1]{\operatorname{#1}}
\newcommand{\F}{\mathbb{F}}

\newtheorem{thm}{Theorem}[section]

\newtheorem{lem}[thm]{Lemma}

\newtheorem{prop}[thm]{Proposition}

\newtheorem{conj}[thm]{Conjecture}
\newtheorem{defn}[thm]{Definition}

\newtheorem*{rem}{Remark}
\newtheorem*{oldproof}{Proof}
\renewenvironment{proof}[1][{}]{\begin{oldproof}[#1]}{\qed\end{oldproof}}

\numberwithin{equation}{section}
\begin{document}

\title{Remarks on Hilbert's tenth problem and the Iwasawa theory of Elliptic curves}

\author{Anwesh Ray}
\address{Department of Mathematics,
University of British Columbia,
Vancouver BC, Canada V6T 1Z2.}
\email{anweshray@math.ubc.ca}


\begin{abstract}
Let $E$ be an elliptic curve with positive rank over a number field $K$ and let $p$ be an odd prime number. Let $K_{\op{cyc}}$ be the cyclotomic $\Z_p$-extension of $K$ and $K_n$ denote its $n$-th layer. The Mordell--Weil rank of $E$ is said to be \emph{constant} in the cyclotomic tower of $K$ if for all $n$, the rank of $E(K_n)$ is equal to the rank of $E(K)$. We apply techniques in Iwasawa theory to obtain explicit conditions for the rank of an elliptic curve to be constant in the above sense. We then indicate the potential applications to Hilbert's tenth problem for number rings.
\end{abstract}



\maketitle
\section{Introduction}\label{s:1}
\par Hilbert's tenth problem for $\Z$ states that there is no Turing machine that takes as input polynomial equations over $\Z$ and decides whether they have nontrivial solutions. Matiyasevich \cite{matiyasevich1970diophantineness} resolved Hilbert's tenth problem over $\Z$. It is well known that if $\Z$ is diophantine as a subset of $\cO_K$, then, the analogue of Hilbert's tenth problem has a negative solution over $\cO_K$. The following conjecture is due to Denef and Lipschitz \cite{denef1978diophantine}.

\begin{conj}[Denef-Lipshitz]\label{main conjecture DL}
For every number field $K$, $\Z$ is a diophantine subset of $\mathcal{O}_K$. 
\end{conj}

There are various special cases in which the above conjecture has been resolved. Conjecture \ref{main conjecture DL} is known to be true for all number fields $K$ such that
\begin{itemize}
    \item $K$ is either totally real or a quadratic extension of a totally real number field (cf. \cite{denef1980diophantine,denef1978diophantine}), 
    \item $K$ has exactly one complex place (cf. \cite{pheidas1988hilbert, shlapentokh1989extension, videla16decimo}),
    \item $K/\Q$ is abelian (cf. \cite{shapiro1989diophantine}).
\end{itemize}

It is natural to study the validity of Conjecture \ref{main conjecture DL} for naturally occurring families of number fields $K$. For instance, Garcia-Fritz and Pasten \cite{garcia2020towards} prove the above conjecture for number fields of the form $\Q(p^{1/3} , \sqrt{-q})$ where $p$ and $q$ range through certain explicit sets of primes of positive Dirichlet density. In this paper, we study Conjecture \ref{main conjecture DL} for certain towers of number field extensions of a fixed number field $K$. More precisely, let $K$ be a number field and $p$ be an odd prime number. Set $K_{\op{cyc}}$ to denote the \emph{cyclotomic $\Z_p$-extension} of $K$. The unique subfield of $K_{\op{cyc}}$ that is of degree $p^n$ over $K$ is denoted by $K_n$ and is called the \emph{$n$-th layer}.
\par Iwasawa studied the growth of the $p$-primary part of the class group of $K_n$ as a function of $n$ (cf. \cite{iwasawa1973zl}). In section \ref{s 3} of this article, we prove that $\mathcal{O}_K$ is a diophantine subset of $\mathcal{O}_{K_n}$ for all $n$, provided there exists an elliptic curve $E_{/K}$ satisfying certain specific additional conditions. The main result is Theorem \ref{main}, which shows that if there exists a suitable elliptic curve $E_{/K}$, then $\mathcal{O}_K$ is a diophantine subset of $\mathcal{O}_{K_n}$ for all $n$. It follows from this that if Conjecture \ref{main conjecture DL} is satisfied for $K$, then it is satisfied for $K_n$ for all $n$.
\par In section \ref{s 4}, we fix an elliptic curve $E_{/\Q}$ of positive Mordell--Weil rank and also fix a number field $K$. We provide circumstantial evidence to show that there is a set of primes $p$ of positive lower density for which the conditions of Theorem \ref{main} are satisfied. This expectation (Conjecture \ref{Omega conj})is based on computational evidence for the behaviour of the \emph{$p$-adic regulator} of an elliptic curve.

\par \emph{Acknowledgment:} The author thanks the anonymous referee for helpful suggestions.

\section{Preliminaries and Notation}\label{s:2}
\par The contents of this section are preliminary in nature. In $\S$\ref{s 2.1}, we introduce the notion of an \emph{integrally Diophantine extension} of number rings, and its implications to Hilbert's tenth problem for number fields. We then in $\S$\ref{s 2.2} recall basic concepts from the Iwasawa theory of elliptic curves.
\subsection{Integrally Diophantine Extensions}\label{s 2.1}
\par Let $A$ be a commutative ring and $n>0$ be an integer. Let $A^n$ be free $A$-module of rank $n$ consisting of elements of the form $a=(a_1, \dots, a_n)$, with entries $a_i\in A$. For $m\geq 0$, $a=(a_1,\dots, a_n)\in A^n$, $b=(b_1, \dots, b_m)\in A^m$, the element $(a,b)\in A^{n+m}$ is given by $(a_1, \dots, a_n, b_1, \dots, b_m)$. Given a set of polynomials $F_1, \dots, F_k\in A[x_1,\dots, x_n, y_1, \dots, y_m]$, and $a\in A^n$ set
\[\cF(a;F_1, \dots, F_k):=\{b\in A^m\mid F_i(a,b)=0\text{ for all }1\leq i\leq k\}.\]
Given a number field $K$, denote by $\cO_K$ its ring of integers.
\begin{defn}
Let $S$ be a subset of $A^n$. The set $S$ is \emph{diophantine} in $A^n$ if for some $m\geq 0$, there are polynomials $F_1, \dots, F_k\in A[x_1,\dots, x_n, y_1, \dots, y_m]$ such that
\[S=\{a\in A^n\mid \cF(a;F_1, \dots, F_k)\text{ is not empty}\}.\]
An extension of number fields $L/K$ is said to be integrally diophantine if $\cO_K$ is a diophantine subset of $\cO_L$. 
\end{defn}
It follows from standard arguments that if $L/\Q$ is an integrally diophantine extension of number fields, then Hilbert's tenth problem has a negative answer for $\cO_L$ (see \cite[p.385]{denef1978diophantine}). Moreover, if $L/F$ and $F/K$ are integrally diophantine extensions of number fields, then so is $L/K$.
\par We recall Shlapentokh's criterion for an extension of number fields to be integrally diophantine.
\begin{thm}[Shlapentokh \cite{shlapentokh2008elliptic}]\label{shlap thm}
Let $L/K$ be an extension of number fields. Suppose there is an elliptic curve $E_{/K}$ such that $\op{rank} E(L)=\op{rank} E(K)>0$. Then, $L/K$ is an integrally diophantine extension. 
\end{thm}

\subsection{Iwasawa theory of elliptic curves}\label{s 2.2}
\par Fix an algebraic closure $\bar{\Q}$ of $\Q$. Let $p$ be an odd prime number and $K$ be a number field. Throughout, $\Z_p$ is the ring of $p$-adic integers. Let $E$ be an elliptic curve defined over $K$ with good ordinary reduction at the primes above $p$. Denote by $\mu_{p^n}$ the group of $p^n$-th roots of unity in $\bar{\Q}$, and set $\mu_{p^\infty}$ to denote the union $\bigcup_n \mu_{p^n}$. We denote by $K(\mu_{p^n})$ (resp. $K(\mu_{p^\infty})$) the cyclotomic extension of $K$ in $\bar{\Q}$ generated by $\mu_{p^n}$ (resp. $\mu_{p^\infty}$). The \emph{cyclotomic $\Z_p$-extension} of $K$ is the unique $\Z_p$-extension of $K$ which is contained in $K(\mu_{p^\infty})$, and is denoted by $K_{\op{cyc}}$. Given a number field extension $\cF$ of $K$, let $\op{Sel}_{p^\infty}(E/\cF)$ denote the $p$-primary Selmer group of $E$ over $\cF$ (see \cite[p.9]{coates2000galois} for the definition).
\par We let $\Sh(E/K)$ denote the Tate-Shafarevich group of $E$ over $K$ and $\Sh(E/K)[p^\infty]$ its $p$-primary part. The $p$-primary Selmer group fits into a short exact sequence
\begin{equation}\label{selmer ses}
    0\rightarrow E(K)\otimes \Q_p/\Z_p\rightarrow \op{Sel}_{p^\infty}(E/K)\rightarrow \Sh(E/K)[p^\infty]\rightarrow 0.
\end{equation}Note that $\op{Sel}_{p^\infty}(E/\cF)$ is a module over $\Z_p[G]$, when $\cF$ is a finite Galois extension of $K$ with Galois group $G=\op{Gal}(\cF/K)$. For $n\geq 0$, let $K_n$ be the unique extension of $K$ in $K_{\op{cyc}}$ such that $\op{Gal}(K_n/K)\simeq \Z/p^n\Z$. The Selmer group of $E$ over $K_{\op{cyc}}$ is taken to be the natural direct limit with respect to restriction maps
\[\op{Sel}_{p^\infty}(E/K_{\op{cyc}}):=\varinjlim_n \op{Sel}_{p^\infty}(E/K_{n}).\] Setting $\Gamma:=\op{Gal}(K_{\op{cyc}}/K)$, note that $\Gamma^{p^n}=\op{Gal}(K_{\op{cyc}}/K_n)$. The Iwasawa algebra is taken to the completed group ring
\[\Lambda(\Gamma):=\varprojlim_n \Z_p[\Gamma/\Gamma^{p^n}].\] Let $\Z_p\llbracket T\rrbracket$ denote the formal power series ring over $\Z_p$ in the variable $T$. Fix a topological generator $\gamma$ of $\Gamma$, and consider the isomorphism $\Lambda(\Gamma)\simeq \Z_p\llbracket T\rrbracket$ sending $\gamma-1$ to the formal variable $T$. The Selmer group $\op{Sel}_{p^\infty}(E/K_{\op{cyc}})$ is naturally a module over the Iwasawa algebra $\Lambda(\Gamma)$. \par Given a module $M$ over $\Lambda(\Gamma)$ let $M^{\vee}$ be its Pontryagin dual $\op{Hom}_{\Z_p}\left(M, \Q_p/\Z_p\right)$. It is conjectured by Mazur that the dual Selmer group $\op{Sel}_{p^\infty}(E/K_{\op{cyc}})^{\vee}$ is a finitely generated and torsion module over $\Lambda$. This is known to be true in the following special cases
\begin{enumerate}
    \item the $p$-primary Selmer group $\op{Sel}_{p^\infty}(E/K)$ (over $K$) is finite (cf. \cite[Theorem 2.8]{coates2000galois}), 
    \item $K$ is an abelian extension of $\Q$. This is a result of Kato (cf. \cite{kato2004p} and \cite[Theorem 2.2]{hachimori1999analogue}). Rubin proved the result for CM elliptic curves.
\end{enumerate}

\par We say that a polynomial in $\Z_p\llbracket T\rrbracket$ is distinguished if it is monic and all its non-leading coefficients are divisible by $p$. A map of $\Z_p\llbracket T\rrbracket$-modules $M_1\rightarrow M_2$ is said to be a \emph{pseudo-isomorphism} if the kernel and cokernel have finite cardinality. One associates a characteristic element and Iwasawa invariants to a finitely generated and torsion module $M$ over $\Lambda(\Gamma)$ as follows. By the structure theorem of finitely generated and torsion $\Z_p\llbracket T\rrbracket$-modules (cf. \cite[Chapter 13]{washington1997introduction}), there is a pseudo-isomorphism 
\[M\longrightarrow \left(\bigoplus_{j=1}^t \frac{\Z_p\llbracket T\rrbracket}{\left(f_j\right)}\right),\] where $f_1, \dots, f_t$ are non-zero elements of $\Z_p\llbracket T\rrbracket$. The characteristic element is the product $\prod_j f_j$ and is denoted by $f_M(T)$. It is well defined up to multiplication by a unit in $\Z_p\llbracket T\rrbracket$. According to the Weierstrass preparation theorem, we may factor $f_M(T)$ as $p^{\mu}P(T)u(T)$, where $\mu\geq 0$, $P(T)$ is a distinguished polynomial and $u(T)$ is a unit in $\Z_p\llbracket T\rrbracket$. The Iwasawa $\mu$ invariant $\mu(M)$ is the quantity $\mu$ that appears in this factorization. The $\lambda$-invariant $\lambda(M)$ is the degree of the polynomial $P(T)$. 

\par Assume that Mazur's conjecture is satisfied for the Selmer group of $E$ over $K^{\op{cyc}}$, i.e., $\op{Sel}_{p^\infty}(E/K^{\op{cyc}})^\vee$ is finitely generated and torsion as a $\Lambda(\Gamma)$. Then, we set $\mu_p(E/K)$ and $\lambda_p(E/K)$ to denote the associated $\mu$ and $\lambda$-invariants for the dual Selmer group $\op{Sel}_{p^\infty}(E/K^{\op{cyc}})^{\vee}$.

\begin{prop}\label{rank <= lambda}
Let $E$ be an elliptic curve over a number field $K$ and let $p$ be an odd prime number. Assume that
\begin{enumerate}
    \item $E$ has good ordinary reduction at all primes $v|p$ of $K$.
    \item The dual Selmer group $\op{Sel}_{p^\infty}(E/K^{\op{cyc}})^{\vee}$ is a finitely generated and torsion module over $\Lambda(\Gamma)$.
\end{enumerate} Then for all $n\geq 0$, we have that $\op{rank} E(K_n)\leq \lambda_p(E/K)$. In particular, $\op{rank} E(K_n)$ is bounded as $n\rightarrow \infty$.
\end{prop}
\begin{proof}
From \eqref{selmer ses} we arrive at the following inequality \[\op{rank} E(K_n)\leq \op{rank}_{\Z_p} \left(\op{Sel}_{p^\infty}(E/K^{n})^{\vee}\right).\]
According to \cite[Theorem 1.9]{greenberg1999iwasawa}, 
\[\op{rank}_{\Z_p} \left(\op{Sel}_{p^\infty}(E/K^{n})^{\vee}\right)\leq \lambda_p(E/K),\] provided $\op{Sel}_{p^\infty}(E/K^{\op{cyc}})^{\vee}$ is finitely generated and torsion as a module over $\Lambda(\Gamma)$. The result follows.
\end{proof}
\section{Hilbert's tenth problem in the cyclotomic $\Z_p$-extension of a number field}\label{s 3}
\subsection{Rank constancy in cyclotomic towers}
\par We shall study the following property in the context of the cyclotomic $\Z_p$-extension of a number field $K$.
\begin{defn}\label{def 3.1}
Let $K$ be a number field and $p$ be a prime number. Let $K_\infty$ be an infinite pro-$p$ extension of $K$. We say that $K$ is integrally diophantine in $K_\infty$ if for all intermediate number fields $L$ such that $K\subseteq L\subseteq K_\infty$, the extension $L/K$ is integrally diophantine.
\end{defn}

We specialize the above notion to $K_\infty=K_{\op{cyc}}$. Thus, $K$ is integrally diophantine in $K_{\op{cyc}}$ precisely when $K_n/K$ is an integrally diophantine extension for all $n\geq 0$. Thus, if $K/\Q$ is integrally diophantine and $K_{\op{cyc}}/K$ is integrally diophantine, then it will follow that $K_n/\Q$ is intergrally diophantine for all $n$. In particular, this shall imply that Hilbert's tenth problem has a negative solution for the ring of integers of all number fields $K_n$ in the infinite cyclotomic tower.
\par We shall use methods from the Iwasawa theory of elliptic curves to derive conditions for a number field $K$ to be integrally diophantine in $K^{\op{cyc}}$.
\begin{prop}
Let $E$ be an elliptic curve over a number field $K$ and let $p$ be an odd prime number. Assume that
\begin{enumerate}
    \item $\op{rank} E(K)>0$,
    \item $E$ has good ordinary reduction at all primes $v|p$ of $K$.
    \item The dual Selmer group $\op{Sel}_{p^\infty}(E/K^{\op{cyc}})^{\vee}$ is a finitely generated and torsion module over $\Lambda(\Gamma)$.
\end{enumerate} Then, there is a value $n_0\geq 0$ such that for all $n\geq n_0$, $K_n/K_{n_0}$ is an integral diophantine extension. Thus, the extension $K_{\op{cyc}}/K_{n_0}$ is integrally diophantine in the sense of Definition \ref{def 3.1}. Therefore if Conjecture \ref{main conjecture DL} is satisfied for $K_{n_0}$, then it is satisfied for $K_n$ for all $n\geq n_0$.
\end{prop}
\begin{proof}
According to Proposition \ref{rank <= lambda}, we find that $\op{rank} E(K_n)$ is bounded as a function of $n$. Therefore, there exists $n_0\geq 0$ such that $\op{rank} E(K_n)=\op{rank} E(K_{n_0})$ for all $n\geq n_0$. The result follows from Theorem \ref{shlap thm}.
\end{proof}
\begin{thm}\label{thm 3.3}
Let $p$ be an odd prime and $E$ be an elliptic curve over a number field $K$ with good ordinary reduction at all primes $v|p$. Assume that $\op{Sel}_{p^\infty}(E/K_{\op{cyc}})^\vee$ is finitely generated and torsion as a module over $\Lambda(\Gamma)$ as conjectured by Mazur. Suppose that \[\lambda_p(E/K)=\op{rank} E(K)>0.\] Then, for all $n\geq 0$, $K_n/K$ is a diophantine extension. Therefore if Conjecture \ref{main conjecture DL} is satisfied for $K$, then it is satisfied for $K_n$ for all $n\geq 0$.
\end{thm}
\begin{proof}
Suppose $K_n/K$ is an integral diophantine extension and $K/\Q$ is an integral diophantine extension. Then, $K_n/\Q$ is an integral diophantine extension and thus Hilbert's tenth problem is negative for $\mathcal{O}_{K_n}$. We show that the above hypotheses imply that $K_n/K$ is an integral diophantine extension for all $n\geq 0$.
\par According to Proposition \ref{rank <= lambda}, we have that $\op{rank} E(K_n)\leq \lambda_p(E/K)$. Since it is assumed that $\op{rank} E(K)\leq \lambda_p(E/K)>0$, it follows that 
\[\op{rank} E(K_n)=\op{rank} E(K)>0.\] Therefore by Theorem \ref{shlap thm}, $K_n/K$ is an integral diophantine extension for all $n\geq 0$.
\end{proof}
\subsection{An Euler characteristic formula}
\par Let $E$ be an elliptic curve over a number field $K$ and let $p$ be a prime number. Assume that $E$ has good ordinary reduction at all primes of $K$ that lie above $p$. Assume that the dual Selmer group $\op{Sel}_{p^\infty}(E/K_{\op{cyc}})^{\vee}$ is a finitely generated torsion $\Lambda(\Gamma)$-module and denote by $f_{E,p}(T)$ its characteristic element. Note that $f_{E,p}(T)$ is well defined up to multiplication by a unit in $\Lambda(\Gamma)$. We may express $f_{E,p}(T)$ as a formal power series in $T$, as follows
\[f_{E,p}(T)=\sum_{i=r}^{\infty}a_i T^i,\] where $a_r\neq 0$. The quantity $r\geq 0$ is the order of vanishing of $f_{E,p}(T)$ at zero, thus we may denote it by $\op{ord}_{T=0} f_{E,p}(T)$. We call $a_r$ the leading coefficient. Perrin-Riou \cite{perrin1994theorie} and Schneider \cite{schneider1985p} provide an explicit formula for the leading coefficient $a_r$. Note that $a_r$ is well defined up to a unit in $\Z_p$. Given two elements $a,b\in \Q_p$, we write $a\sim b$ to mean that $a=ub$ for a unit $u\in \Z_p^\times$. The $p$-adic regulator $\mathcal{R}_p(E/K)$ is the determinant of the $p$-adic height pairing on the Mordell-Weil group of $E$. The reader is referred to \cite{ schneider1982p,schneider1985p, mazur1986p} for further details. Schneider has conjectured that the $p$-adic regulator $\mathcal{R}_p(E/K)$ is always non-zero, i.e., the $p$-adic height pairing is always non-degenerate. We set $R_p(E/K)$ to be the normalized $p$-adic regulator, defined by
$R_p(E/K):=p^{-\op{rank}E(K)}\mathcal{R}_p(E/K)$. At each prime $v$ of $K$, let $c_v(E/K)$ be the Tamagawa number of $E$ at $v$ (cf. \cite[Chapter 3]{coates2000galois} for the definition). If $v$ is a prime at which $E$ has good reduction, then $c_v(E/K)=1$. At any prime $v$ of $K$, let $\F_v$ denote the residue field of $\mathcal{O}_v$ at $v$. For each prime $v|p$, let $\widetilde{E}$ be the reduction of $E$ to an elliptic curve over $\F_v$.
\begin{thm}[Perrin Riou, Schneider]\label{schneider thm}
Let $E$ be an elliptic curve over a number field $K$ and $p$ an odd prime. Assume that the following conditions are satisfied
\begin{enumerate}
    \item $E$ has good ordinary reduction at all primes $v|p$ of $K$,
    \item $\Sh(E/K)[p^\infty]$ is finite,
    \item $\mathcal{R}_p(E/K)$ is non-zero,
    \item $\op{Sel}_{p^\infty}(E/K_{\op{cyc}})^\vee$ is finitely generated and torsion as a $\Lambda(\Gamma)$-module.
\end{enumerate}
Then, the following assertions hold:
\begin{enumerate}
    \item $r:=\op{ord}_{T=0} f_{E,p}(T)=\op{rank} E(K)$,
    \item the leading coefficient is given up to a unit by \[a_r\sim\frac{R_p(E/K)\times \#\Sh(E/K)[p^\infty]\times \prod c_v(E/K)\times \left(\prod_{v|p}  \#\widetilde E(\F_v)[p^\infty]\right)^2}{\left(\# E(K)[p^\infty]\right)^2}.\]
\end{enumerate}
\end{thm}

\begin{proof}
The above result is \cite[Theorem 2', p.342]{schneider1985p}.
\end{proof}

\begin{lem}\label{lemma rank=lambda}
Let $E$ be an elliptic curve satisfying the conditions of Theorem \ref{schneider thm} and set $r:=\op{ord}_{T=0}f_{E,p}(T)=\op{rank} E(K)$. The following conditions are equivalent
\begin{enumerate}
    \item $a_r$ is a unit in $\Z_p$,
    \item $\mu_p(E/K)=0$ and $\lambda_p(E/K)=\op{rank} E(K)$.
\end{enumerate}
\end{lem}
\begin{proof}
We write $f_{E,p}(T)$ as $T^rg(T)$ where $g(0)=a_r\neq 0$. Suppose that $a_r$ is a unit in $\Z_p$. Then, $g(T)$ is a unit in $\Lambda(\Gamma)$ and the Weierstrass factorization of $f_{E,p}(T)$ is given by $P(T)u(T)$, where $P(T)=T^r$ is a distinguished polynomial and $u(T)=g(T)$ is a unit in $\Lambda(\Gamma)$. Thus, $\lambda_p(E/K)=\op{deg}P(T)=\op{deg}T^r=r$, and $\mu_p(E/K)=0$.
\par On the other hand, suppose that $\mu_p(E/K)=0$ and $\lambda_p(E/K)=r$. Then, we may write $f_{E,p}(T)=P(T)u(T)$, where $P(T)$ is a distinguished polynomial of degree $r$ and $u(T)$ is a unit in $\Lambda(\Gamma)$. Since $P(T)$ is divisible by $T^r$, this forces $P(T)=T^r$ and $u(T)=g(T)$. Therefore, $g(T)$ is a unit in $\Lambda(\Gamma)$, and hence $a_r=g(0)$ is unit in $\Z_p$.
\end{proof}
\begin{thm}\label{main}
Let $E$ be an elliptic curve over a number field $K$ and $p$ an odd prime. Assume that the following conditions are satisfied
\begin{enumerate}
    \item $E$ has good ordinary reduction at all primes $v|p$ of $K$,
    \item $\op{rank} E(K)>0$,
    \item $\op{Sel}_{p^\infty}(E/K_{\op{cyc}})^\vee$ is finitely generated and torsion as a $\Lambda(\Gamma)$-module,
    \item $R_p(E/K)$ is a $p$-adic unit (in particular, nonzero), \item $\Sh(E/K)[p^\infty]=0$,
    \item $p\nmid c_v(E/K)$ for all primes $v$ at which $E$ has bad reduction,
    \item $p\nmid \#\widetilde E(\F_v)$ for all primes $v|p$ of $K$.
\end{enumerate}
Then, for all $n\geq 0$, $K_n/K$ is an integrally diophantine extension. Therefore if Conjecture \ref{main conjecture DL} is satisfied for $K$, then it is satisfied for $K_n$ for all $n\geq 0$.
\end{thm}

\begin{proof}
Note that since it is assumed that $\Sh(E/K)[p^\infty]=0$, in particular, $\Sh(E/K)[p^\infty]$ is finite. The conditions of Theorem \ref{schneider thm} are satisfied. It follows from Theorem \ref{schneider thm} that $a_r$ is a unit, where $r=\op{rank} E(K)$. Lemma \ref{lemma rank=lambda} then asserts that $\lambda_p(E/K)=\op{rank} E(K)$. Since $\op{rank} E(K)>0$, the result follows from Theorem \ref{thm 3.3}.
\end{proof}

The above result thus shows that given a number field $K$, if there exists an elliptic curve $E_{/K}$ satisfying the conditions of Theorem \ref{main}, then, $K_n/K$ is integrally diophantine for all $n$. In the next section, we shall better explain the conditions of Theorem \ref{main}.
\section{Conditions for rank constancy in cyclotomic $\Z_p$-towers}\label{s 4}
\par Let $E$ be an elliptic curve defined over $\mathbb{\Q}$ such that $\op{rank} E(\Q)>0$, and let $K/\Q$ be a number field extension. We consider the base change of $E$ to $K$. Note that $\op{rank} E(K)>0$. The data $(E,K)$ is fixed throughout and the results are of most interest in the case when $K$ is neither totally real, nor abelian. Assume that $E$ does not have complex multiplication. Given any set of prime numbers $\mathcal{P}$, the upper (resp. lower) density of $\mathcal{P}$ shall refer to its upper (resp. lower) Dirichlet density. When we say that $\mathcal{P}$ has density $\delta\in [0,1]$, we mean that the Dirichlet density of $\mathcal{P}$ exists and equals $\delta$. Note that the upper and lower densities always exist. Let $\Omega$ be the set of odd prime numbers $p$ such that $E$ has good ordinary reduction at $p$ and the conditions of Theorem \ref{main} are satisfied for $E_{/K}$ and the cyclotomic $\Z_p$-extension of $K$. In this section, we provide circumstantial evidence for the following conjecture.
\begin{conj}\label{Omega conj}
The set of primes $\Omega$ has positive lower density.
\end{conj}
In this section, the prime $p$ is allowed to vary over the prime numbers at which $E$ has good ordinary reduction. A classical result of Serre shows that for a non-CM elliptic curve, the set of primes of good ordinary (resp. supersingular) reduction has density $1$ (resp. $0$). Since the prime $p$ is not fixed in this section, it is pertinent that we do not suppress the role of $p$ in the notation for cyclotomic extension. Thus, we shall set $K_{\op{cyc}}^{(p)}$ to denote the cyclotomic $\Z_p$-extension of $K$ and $K_{n}^{(p)}$ the subfield of $K_{\op{cyc}}^{(p)}$ with $[K_n:K]=p^n$.
\begin{thm}
Let $K$ be a number field and $E$ be a non-CM elliptic curve over $\Q$ such that
\begin{enumerate}
    \item $\op{rank} E(\Q)>0$,
    \item $\Sh(E/K)$ is finite.
\end{enumerate}
Then there exists a positive density set of odd prime numbers $\Sigma$ satisfying the following conditions
\begin{enumerate}
    \item $E$ has good ordinary reduction at all prime numbers $p\in \Sigma$,
    \item if for a prime $p\in \Sigma$, the following conditions are satisfied
    \begin{enumerate}
        \item $R_p(E/K)$ is a $p$-adic unit,
        \item $\op{Sel}_{p^\infty}(E/K_{\op{cyc}})^{\vee}$ is a finitely generated and torsion module over $\Lambda(\Gamma)$,
    \end{enumerate} then, $\op{rank} E(K_n^{(p)})=\op{rank} E(K)$ for all $n$ and the conclusion of Theorem \ref{main} holds.
\end{enumerate}
\end{thm}

\begin{proof}
By the aforementioned result of Serre, the set of primes $p$ at which $E$ has good ordinary reduction has density $1$. Let $\widetilde{K}\subset \bar{\Q}$ be the Galois closure of $K$. By a standard application of the Chebotarev density theorem, the set of prime numbers $p$ that split completely in $\widetilde{K}$ has density $[\widetilde{K}:\Q]^{-1}$. Let $\Sigma_0$ be the set of primes $p$ at which $E$ has good ordinary reduction that are completely split in $K$. We find that $\Sigma_0$ has density $[\widetilde{K}:\Q]^{-1}$.
\par Let the set of primes $p\in \Sigma_0$ such that $p$ divides $\#\widetilde{E}(\F_p)$ be denoted by $\Sigma_1$. Primes $p$ at which $E$ has good reduction such that $p$ divides $\#\widetilde{E}(\F_p)$ are known as \emph{anomalous} primes. If $p$ is large enough, then a prime $p$ is anomalous precisely when $a_p(E):=p+1-\#\widetilde{E}(\F_p)$ is equal to $1$. As is well known, the set of anomalous primes has density $0$ (cf. \cite{murty1997modular}). It follows that $\Sigma_1$ has density $0$. Note that since any prime $p\in \Sigma_0\backslash\Sigma_1$ is completely split in $\widetilde{K}$, for each prime $v$ of $K$ that lies above $p$, we find that $\F_v=\F_p$. Therefore, since $p$ is not anomalous, we find that $p$ does not divide the product $\prod_{v\mid p} \# \widetilde{E}(\F_v)$ for all primes $p\in \Sigma_0\backslash\Sigma_1$. The set of primes that divide any given natural number is clearly finite. Therefore, the set of primes $p\in \Sigma_0$ such that $p\mid \prod_v c_v(E/K)$ is finite. Denote this set by $\Sigma_2$. Finally, note that since it is assumed that $\Sh(E/K)$ is finite, the set of primes $p$ such that $\Sh(E/K)[p^\infty]\neq 0$ is finite as well. Let $\Sigma_3$ denote this set. Set $\Sigma$ to be the set of primes $\Sigma_0\backslash \left(\bigcup_{i=1}^3 \Sigma_i\right)$. We find that for $p\in \Sigma$, if $R_p(E/K)$ is a unit in $\Z_p$, then, the conditions of Theorem \ref{main} are satisfied and the result follows.
\end{proof}

One would like to understand how often $R_p(E/K)$ is \emph{not} a $p$-adic unit. Unfortunately, the $p$-adic regulator is a fairly complex invariant computations involving the $p$-adic regulator over a number field $K$ are challenging.

\begin{rem}There are built in packages on the Sage computational system \cite{stein2008sage} to compute $p$-adic regulators over $\Q$, however, the author is not aware of any such packages written for number fields $K\neq \Q$ that are readily usable.
\end{rem}Computations over $\Q$ suggest the following conjecture over other number fields.
\begin{conj}
Let $E$ be an elliptic curve over $\mathbb{Q}$ and $K$ be a number field. The set of primes $p$ such that $R_p(E/K)$ is divisible by $p$ has density $0$.
\end{conj}

Given an elliptic curve $E_{/\Q}$ and $N>0$, let $\Pi^{\leq N}$ be the set of primes $p\leq N$ of good ordinary reduction such that $p\mid R_p(E/\Q)$. We compute $\Pi^{\leq 1000}$ for the first ten elliptic curves of rank $2$ ordered by conductor.

The calculations in the following table are from \cite[p.7955]{kundu2021statistics} and were done on Sage.
\vspace{1.0cm}
\begin{center}
\begin{tabular}{c|c|c |c|c|c} 
 & Cremona Label & $\Pi^{\leq 1000}$ & & Cremona Label & $\Pi^{\leq 1000}$\\ [1 ex]
 \hline
 $1.$ & $389a$ & $\emptyset$ & $6.$ & $643a$ & $\emptyset$\\
 $2.$ & $433a$ & $\{13\}$& $7.$ & $655a$ & $\{7,31\}$\\
 $3.$ & $446d$ & $\{7\}$&$8.$ & $664a$ & $\{59\}$\\
 $4.$ & $563a$ & $\emptyset$&$9.$ & $681c$ & $\emptyset$ \\
 $5.$ & $571b$ & $\emptyset$&$10.$ & $707a$& \{29\}.\\
\end{tabular}
\end{center}
\vspace{1cm}

The data indicates that the $p\mid R_p(E/K)$ should be a rare occurrence.

\par We give one concrete example of an elliptic curve $E$ over $\Q$ for which Theorem \ref{main} can be expected to apply. We do not obtain any new result by working over $\mathbb{Q}$, the calculation below is only included to illustrate our technique.
\par Consider the elliptic curve $E$ with Weierstrass equation $y^2+y=x^3-x$, Cremona label 37a1, and set $p=5$. We discuss the conditions of Theorem \ref{main} below,
\begin{enumerate}
    \item $E$ has good ordinary reduction at $5$, 
    \item $\op{rank}E(\Q)=1$, in particular, is $>0$. 
    \item We assume that the dual Selmer group $\op{Sel}_{5^\infty}(E/\Q_{\op{cyc}})^{\vee}$ is torsion over the Iwasawa algebra. Recall that Mazur's conjecture predicts this condition to hold for all elliptic curves with good ordinary reduction at $p$.
    \item The normalized $5$-adic regulator is a $5$-adic unit.
    \item The analytic order of $\Sh(E/\Q)$ is exactly $1$. 
    \item All Tamagawa numbers $c_v(E/\Q)$ are equal to $1$.
    \item We find that $\widetilde{E}(\F_5)=8$, and hence, is not divisible by $5$.
\end{enumerate}
We note here that the calculations above were performed via Sage, cf. \cite{stein2008sage}.

\bibliographystyle{abbrv}
\bibliography{references}

\end{document}